\documentclass[11pt,leqno]{amsart}
\usepackage{amssymb}
\usepackage{graphicx}
\usepackage{makecell}
\numberwithin{equation}{section}
\usepackage{fullpage}

\usepackage{hyperref}
\allowdisplaybreaks[2]

\renewcommand\b{\beta}

\def\eps{\varepsilon }

\renewcommand\b{\beta}

\def\eps{\varepsilon}

\newcommand\br{\begin{remark}}
\newcommand\er{\end{remark}}
\newcommand\bp{\begin{pmatrix}}
\newcommand\ep{\end{pmatrix}}
\newcommand{\be}{\begin{equation}}
\newcommand{\ee}{\end{equation}}
\newcommand\ba{\begin{equation}\begin{aligned}}
\newcommand\ea{\end{aligned}\end{equation}}

\newcommand{\bap}{\begin{app}}
\newcommand{\eap}{\end{app}}
\newcommand{\begs}{\begin{exams}}
\newcommand{\eegs}{\end{exams}}
\newcommand{\beg}{\begin{example}}
\newcommand{\eeg}{\end{example}}
\newcommand{\bpr}{\begin{proposition}}
\newcommand{\epr}{\end{proposition}}
\newcommand{\bt}{\begin{theorem}}
\newcommand{\et}{\end{theorem}}
\newcommand{\bc}{\begin{corollary}}
\newcommand{\ec}{\end{corollary}}
\newcommand{\bl}{\begin{lemma}}
\newcommand{\el}{\end{lemma}}
\newcommand{\bd}{\begin{definition}}
\newcommand{\ed}{\end{definition}}
\newcommand{\brs}{\begin{remarks}}
\newcommand{\ers}{\end{remarks}}

\newcommand{\U }{\mathcal{U}}
\newcommand{\A }{\mathcal{A}}

\newcommand{\Id}{{\rm Id }}

\newtheorem{theorem}{Theorem}[section]
\newtheorem{proposition}[theorem]{Proposition}
\newtheorem{corollary}[theorem]{Corollary}
\newtheorem{lemma}[theorem]{Lemma}

\theoremstyle{remark}
\newtheorem{remark}[theorem]{Remark}
\theoremstyle{definition}
\newtheorem{definition}[theorem]{Definition}

\newtheorem{example}[theorem]{Example}

\newcommand{\beq}{\begin{equation}}
\newcommand{\eeq}{\end{equation}}

\title{Orbital stability of undercompressive viscous shock waves under $L^1\cap H^4$ perturbation}
\author{Zhao Yang}
\address{Academy of Mathematics and Systems Science, Chinese Academy of Sciences, Beijing
100190 China.}
\email{yangzhao@amss.ac.cn}
\thanks{Research of Z.Y. was partially supported by the IU College of Arts and Sciences Dissertation Year Fellowship
and by TODO (mention Urbana and CAS support, excellent 100, etc.)}

\author{Kevin Zumbrun}
\address{Indiana University, Bloomington, IN 47405}
\email{kzumbrun@iu.edu}
\thanks{Research of K.Z. was partially supported
under NSF grants no. DMS-1700279, DMS-2154387, and DMS-2206105}

\begin{document}

\begin{abstract}
By the use of a new vertical estimate introduced by the authors in the context of relaxation shocks
for shallow water flow, we both simplify and extend the basic $L^1\cap H^3$ stability results 
of Mascia and Zumbrun for viscous shock waves, in particular extending their results for Lax waves to the 
undercompressive case.
\end{abstract}

\date{\today}
\maketitle


\tableofcontents

\section{Introduction}\label{s:intro}
In this note, we establish by a single argument orbital stability under $L^1\cap H^4$ perturbation
of arbitrary amplitude viscous Lax or undercompressive shock waves of partially symmetric hyperbolic--parabolic 
systems of conservation laws as defined in \cite{Z1}, under the assumption of spectral, Evans function stability.
These include in particular spectrally stable viscous shock solutions of systems of classical 
Kawashima type \cite{K}.  This (partially) extends previous phase-asymptotic orbital stability analyses 
\cite{ZH,HZ,RZ,HRZ} based on pointwise estimates requiring much stronger localization ($\sim (1+|x|)^{-3/2}$) 
of the initial data.  At the same time, it extends the simpler $L^1\cap H^3$ orbital stability analyses 
of \cite{MZ1,MZ2} for Lax waves to the undercompressive case.

Consider the degenerate parabolic conservation laws
\be 
\label{Geq}
F^0(U)_t+F^1(U)_x=(B(U)U_x)_x,\quad U=\left(\begin{array}{c} U^I\\U^{II}\end{array}\right),\quad B=\left(\begin{array}{rr}0&0\\0&b\end{array}\right),
\ee 
$x,t\in\mathbb{R}$, $U,\;F^0,\;F^1\in \mathbb{R}^n$, $U^{II}\in \mathbb{R}^{r}$, $B\in\mathbb{R}^{n\times n}$, $b\in \mathbb{R}^{r\times r}$, with viscous shock profiles, or (without loss of generality) standing wave solutions $\bar{U}(x)$, $\lim_{x\rightarrow \pm\infty}=\bar{U}_\pm$ satisfying

\medskip

{(A1)} $dF^1(\bar{U}_\pm)$, $dF^1_{11}$, $dF^0$ are symmetric, $dF^0>0$,

{(A2)} no eigenvector of $dF^1{(dF^0)}^{-1}(\bar{U}_\pm)$ lies in the kernel of $B{(dF^0)}^{-1}(\bar{U}_\pm)$,

{(A3)} $b>0$.

\medskip

Following \cite{ZH,MZ1,Z1}, we denote as asymptotic hyperbolic characteristics the eigenvalues
$a_j^\pm$ of limiting matrices $A_\pm:= dF^1{(dF^0)}^{-1}(\bar{U}_\pm)$, ordered by
\ba
\label{characteristics}
&a_1^+<a_2^+<\cdots<a_{i_+}^+<0<a_{i_++1}^+<\cdots<a_n^+,\\
&a_1^-<a_2^-<\cdots<a_{i_-}^-<0<a_{i_-+1}^-<\cdots<a_n^-.
\ea 
The shock is called hyperbolic Lax type if $i_+=i_-+1$ (Lax $i_+$-shock), 
undercompressive type if $i_+\leq i_-$, and overcompressive type if $i_+\geq i_-+2$.
If it is of hyperbolic Lax or undercompressive type, then it is called \cite{MZ1,Z1} ``pure''
viscous Lax or undercompressive type if
also $\bar U$ is, up to translation, the unique traveling-wave profile connecting $\bar{U}_\pm$.

From here on, we restrict discussion to the Lax or undercompressive case, assuming

\medskip

(H1) $\bar U$ is unique up to translation,

\medskip
\noindent
and refer to waves simply as Lax or undercompressive type, omitting the descriptor ``pure''.

As discussed in \cite{MZ1,Z1}, given a viscous shock profile, one may define an {\it Evans function}
$D(\cdot)$ associated with the linearized operator $L$ about the profile (given in \eqref{Leq} below), with
the properties that $D$ is analytic in the closed right half-plane, with nonzero roots $\lambda$ agreeing in
location and multiplicity with the nonzero eigenvalues of $L$. (Under our assumptions, 
$L$ has no essential spectrum in the closed right half-plane except at $\lambda=0$ \cite{MZ1,Z1}.)

In the Lax or overcompressive case considered here, this allows a particularly concise
characterization of linearized stability in the form of 
the {\it Evans function condition} \cite{MZ1,Z1}:

\medskip
($\mathcal{D}$) There exists precisely one zero of $D(\cdot)$ in the nonstable half-plane $\{\lambda:\Re \lambda \geq  0\}$, necessarily at the origin.
\medskip

\noindent
This condition, necessary and sufficient for $L^1\to L^p$ linearized stability \cite{MZ1,Z1} is a generalized
spectral stability criterion encoding both stability of usual point spectra, and the behavior of the resolvent
in the vicinity of the $\lambda=0$ embedded in (the boundary of) the essential spectrum of $L$.

\subsection{Main results}\label{s:main}
Our main result extends to pure undercompressive profiles the nonlinear orbital stability result obtained for
pure Lax profiles in \cite{MZ1,MZ2}.

\bt\label{main1}
Let $\bar{U}$ be a viscous shock profile of Lax or undercompressive type, satisfying (A1)--(A3) and (H1). 
And, assume that $\bar{U}$ is spectrally stable in the sense of the Evans function condition ($\mathcal{D}$).
Then, $\bar{U}$ is nonlinearly orbitally stable with respect to initial perturbations $u_0$ of  sufficiently small norm in $L^1\cap H^{s}$, $s\geq 4$. 
That is, for perturbation $|u_0|_{L^1\cap H^s}\leq \eps$ and initial data $\tilde U_0:=\bar{U}_0+u_0$, 
there exists a global solution $\tilde U$ of \eqref{Geq} and constant $C>0$ satisfying
for $t\geq 0$, $2\leq p\leq \infty$: 
\ba\label{mainests}
|\tilde U(\cdot+\delta(t),t)-\bar{U}(\cdot)|_{H^s}&\leq C\eps (1+t)^{-1/4},\\
|\tilde U(\cdot+\delta(t),t)-\bar{U}(\cdot)|_{L^p}&\leq C\eps (1+t)^{-(1/2)(1-1/p)}, \\
|\tilde U_x(\cdot+\delta(t),t)-\bar{U}_x(\cdot)|_{L^p}&\leq C\eps (1+t)^{-(1/2)(1-1/p)}, \\
|\dot \delta(t)|&\leq C\eps (1+t)^{-(1/2)},\\
|\delta(t)|&\leq C\eps,
\ea
where the phase shift $\delta(t)$ defined in \eqref{delta} is the approximate shock location of $\tilde{U}$. Moreover, for each fixed $x$ there holds the vertical estimate
\be\label{vert}
\int_0^{t} (1+s)^{-1/2}| \tilde U(x+\delta(s),s)-\bar{U}(x)|ds \leq C\eps.
\ee
\et

\br\label{sprmk}
The conclusions of Theorem \ref{main1} hold also in the strictly parabolic case $F^0=\Id$, 
$\Re\,{\rm Spec}(B)>0$, with (A1)-(A3) replaced by hydrodynamic, or ``Majda-Pego'' stability of endstates
\cite{MP}. Indeed, one may in this case reduce the regularity requirement to $H^1$ on initial data.
For, in this case one obtains (stronger versions of) the same linearized estimates as in
Proposition \ref{greenbounds}- see \cite{ZH,HZ}-
while the nonlinear damping estimate of Proposition \ref{damping}
follows trivially from standard parabolic energy estimates using Lyapunov's lemma.
\er

\subsection{Discussion and open problems}\label{s:disc}
Theorem \ref{main1} together with Remark \ref{sprmk} recovers all previous orbital stability results
for pure Lax or undercompressive shocks, with only $L^1$ localization assumed on the data.
It does not recover ``phase-asymptotic'' orbital stability, in which the phase shift $\delta(t)$
is shown to converge to a limiting value as $t\to \infty$, which up to now requires perturbations
localized as $(1+|x|)^{-3/2}$ \cite{HZ,RZ,HRZ}, yielding convergence at rate $t^{-1/2}$.
However, it makes up for this in simplicity/transparency, avoiding the delicate pointwise nonlinear
estimates of \cite{HZ,RZ,HRZ}.
Moreover, the $L^1$ requirement substantially reduces the localization required on the initial perturbation.
And, as noted in \cite{MZ1,MZ2}, 
it is easy to see that there can be no {\it uniform rate of phase convergence} for $L^1$ perturbations,
by considering a sequence of identical compactly supported perturbations centered at $x=n$, with
$n=1, 2, \dots$ going out to infinity, and noticing that these take arbitrarily long to affect the shock,
but eventually do cause a nonzero shift in phase.

Theorem \ref{main1} was proven in the Lax case in \cite{MZ1,MZ2} by the use of ``refined'' linear estimates
valid in the Lax and overcompressive but not the undercompressive case.  
The proof of these refined estimates required considerable extra effort beyond the proof of basic pointwise 
Green function bounds. And, it was not at all clear from that analysis whether 
the conclusions could hold in the undercompressive case; see Remark \ref{degrmk}.
Thus, the present analysis both simplifies the treatment of \cite{MZ1,MZ2} in the Lax case, and
{\it answers
the 20-year open problem of nonlinear orbital stability with respect to $L^1$ perturbations 
in the undercompressive case.}

What replaces the refined linear estimates in our analysis here is the ``vertical estimate'' \eqref{vert} 
of Theorem \ref{main1}.  This in turn follows from the Strichartz-type estimates \eqref{aux1}--\eqref{aux3} of 
Proposition \ref{auxprop} below. 
These Strichartz estimates include a small amount of geometric information involving transversality 
of Gaussian propagator elements, without entering the full pointwise analysis of \cite{HZ,RZ,HRZ},
and this minimal treatment allows us to close the nonlinear stability arguments with initial perturbation 
merely $L^1$ data, rather than the algebraic localization $\sim (1+|x|)^{-3/2}$ of the pointwise references.

Such estimates were introduced in \cite{YZ} in the context of stability of discontinuous shock profiles
of the Saint-Venant equations for inclined shallow water flow, a relaxation system with scalar
equilibrium system, hence of {\it essentially scalar} type, with principal linear modes propagated inward
toward the shock along equilibrium characteristics.
The current analysis shows that these ideas are relevant also in a genuinely system case, for which
signals are propagated also away from the shock, along outgoing characteristic directions.
Specifically,  we find it necessary to use vertical estimates in the proof of the key Lemma \ref{zetalem}
not only to bound phase shift $\delta(s)$ as in \cite{YZ}, but already in the estimate of the interior
perturbation $w(\cdot, s)$: precisely in order to handle new terms coming from outgoing characteristic modes.
(See the italicized comment in the proof of Lemma \ref{zetalem}.)
Thus, we must include the vertical estimate already in the main iteration yielding orbital stability, 
rather than as a bootstrap argument as in \cite{YZ}.
The latter innovations we believe to be the key to showing nonlinear orbital stability also for 
discontinuous shock profiles of general relaxation systems, with {\it nonscalar} equilibrium equations,
resolving the key new difficulty of outgoing equilibrium characteristic modes.

In \cite{YZ}, we obtained the further result of phase-asymptotic orbital stability, with nonuniform rate depending
on the decay of the tail of the initial perturbation.
This was done with the help of an additional ``approximate incoming characteristic estimate,'' 
which however, relied much on the scalar property that all equilibrium characteristics propagated
inward toward the shock.
It is a very interesting open problem whether one could recover phase-asymptotic orbital stability
at the same tail rate also for undercompressive shocks, in the (necessarily, since they otherwise do not exist) 
system case, perhaps by further elaboration of the approximate characteristic estimates of \cite{YZ}.
Likewise, both $L^1\cap H^s$ orbital and phase-asymptotic orbital stability are important open problems for
overcompressive shock waves, and for
discontinuous shock profiles of any type for relaxation systems with nonscalar equilibrium equations.


\section{Preliminaries}\label{s:prelims}
We begin by recalling, essentially verbatim, some basic ingredients established in \cite{MZ1,MZ2}
and \cite{Z1}, along with a key new Strichartz estimate of a type introduced in \cite{YZ}.

\subsection{Nonlinear perturbation equations}\label{s:nlin}

We introduce now the equation
\be\label{Weq}
W_t + \tilde F(W)_x=(\tilde B(W)W_x)_x
\ee
where $W:=F^0(U)$, $\tilde{F}(W):=F^1({(F^0)}^{-1}(W))$, and 
$$
\tilde{B}(W):=B({(F^0)}^{-1}(W)) {(dF^0)}^{-1}(W)= \begin{pmatrix} 0 & 0\\ b_1& b_2 \end{pmatrix}.
$$
Denote the profile in $W$ coordinates by $\overline{W}:=F^0(\bar{U})$. 

Under the assumption (H1), following \cite{MZ2}, we may conveniently work with the ``centered'' perturbation
variable
\begin{equation}\label{pert}
	w(x,t):=F^0(\tilde U(x+\delta(t), t))-F^0(\bar U(x)),
\end{equation}
which satisfies
\begin{equation}
\label{perteq}
w_t-Lw=Q(w,w_x)_x+\dot \delta (t)(\overline{W}_x + w_x),
\end{equation}
where
\ba\label{Leq}
Lw &:= -(Aw)_x
+ \big(\tilde B(\overline{W}) w_x \big)_x,\quad Aw:=d\tilde F(\overline{W})  w
- \big(d\tilde{B}(\overline{W})  w\big) \overline{W}_x,
\ea
is the linearized operator about the wave and
\begin{equation}\label{Q}
\begin{aligned}
Q(w,w_x)&=\mathcal{O}(|w|^2 + |w||w_x|),\\
Q(w,w_x)_x&=\mathcal{O}(|w||w_x|+|w_x|^2 + |w||w_{xx}|)\\
\end{aligned}
\end{equation}
so long as $|w|$, $|w_x|$ and $|w_{xx}|$ remain bounded.
See \cite{MZ2} for further details.

\subsubsection{Green function representation}\label{s:grep}
Let $G$ be the Green function associated with the operator $L$ defined in \eqref{Leq}.
Recalling the standard fact that $\overline{W}'=dF^0(\bar U)\bar U'$ is a stationary
solution of the linearized equations for \eqref{Weq}, so that
$L\overline{W}'=0$, or
$$
\int^\infty_{-\infty}G(x,t;y)\overline{W}'(y)dy=e^{Lt}\overline{W}'(x)
=\overline{W}'(x),
$$
we have by Duhamel's principle:
\ba\label{duhamel}
  w(x,t)&=\int^\infty_{-\infty}G(x,t;y)w_0(y)\,dy  -\int^t_0 \int^\infty_{-\infty} G_y(x,t-s;y)
  \big(Q(w,w_x)+\dot \delta w \big) (y,s)\,dy\,ds\\ &\quad + \delta (t)\overline{W}'(x).
  \ea

\subsection{Linearized estimates}\label{s:lin}
We recall from \cite{MZ1,RZ} the following linearized estimates.\footnote{Though the bounds of \cite{MZ1}
are stated only in the Lax and overcompressive cases, the main part of the proof applies equally
well in the undercompressive case to yield the degraded bounds ($\gamma=1$) stated here. It is only
the ``refined estimates'' $\gamma=0$ of the Lax and overcompressive cases that do not apply here.}

\begin{proposition}[\cite{MZ1,RZ}]\label{greenbounds}
Under the assumptions of Theorem \ref{main1}, the Green function
$G(x,t;y)$ may be decomposed as $G=H+E+\tilde G$, where, for $y\leq  0$,
\be
H(x,t;y):=\sum\limits_{j=1}^{J} a_j^{*-1}(x) a_j^{*}(y)
 {\mathcal R}_j^*(x) \zeta_j^*(y,t) \delta(x-y-\bar a_j^* t)
 {\mathcal L}_j^{*t}(y),
\ee
where $a^{*}_j(x)$, $j=1,\dots,J\le(n-r)$,  denote the eigenvalues of $A_*\in \mathbb{R}^{(n-r)\times(n-r)}$ with multiplicity $m_j^*$ where
\be 
\label{A*}
A_*:=A_{11}-A_{12}b_2^{-1}b_1,\quad\text{with}\quad  A=:\left[\begin{array}{rr}A_{11}&A_{12}\\A_{21}&A_{22}\end{array}\right],\quad \tilde{B}(\overline{W})=:\left[\begin{array}{rr}0&0\\b_1&b_2\end{array}\right],
\ee 
the averaged convection rates $\bar{a}_j^*(x,t)$ denote the time-averages over $[0,t]$ of $a_j^*(z)$ along backward characteristic path $z_j^*=z_j^*(x,t)$ defined by $dz_j^*/dt=a_j^*(z_j^*)$, $z_j^*(t)=x$, the eigenmodes ${\mathcal L}_j^*$, ${\mathcal R}_j^*$ are defined by
$$
{\mathcal L}_j^*:=\left[\begin{array}{c}L^*_j\\0\end{array}\right],\quad {\mathcal R}_j^*:=\left[\begin{array}{c}R^*_j\\-b_2^{-1}b_1R^*_j\end{array}\right]
$$
where $L^*_j$, $R^*_j$ are $(n-r)\times m_j^*$ blocks of left and right eigenvectors of $A_*$ associated to $a^*_j$ satisfying the dynamical normalizing ${L^*_j}^t\partial_xR^*_j\equiv 0$ along with the usual static normalization ${L^*_j}^tR^*_j\equiv \Id_{m_j^*}$, the dissipation matrix $\zeta_j^*(x,t)\in R^{m_j\times m_j}$ is defined by the dissipative flow $d\zeta_j^*/dt=L_j^{*t}D_*R_j^*(z_j^*)\zeta_j^*$, $\zeta_j^*(0)=I_{m_j^*}$ where
$$
D_*(x):=A_{12}b_2^{-1}\left[A_{21}-A_{22}b_2^{-1}b_1+A^*b_2^{-1}b_1+b_2\partial_x(b_2^{-1}b_1)\right](x),
$$
and $\delta$ denotes Dirac mass,
\ba\label{e}
E(x,t;y)&:= \overline{W}'(x)e(y,t),\\
e(y,t)&:=\chi_{_{t\geq 1}}\sum_{a_i^{-}>0}
  \left({\rm errfn}\left(\frac{y+a_i^{-}t}{\sqrt{4\beta_i^{-}t}}\right)
  -{\rm errfn}\left(\frac{y-a_i^{-}t}{\sqrt{4\beta_i^{-}t}}\right)\right)
  l_{i}^{-}(y),
\ea
where ${\rm errfn} (z) := \frac{1}{2\pi} \int_{-\infty}^z e^{-\xi^2} d\xi$ denotes the error function, $a_i^\pm$ as ordered in \eqref{characteristics} are eigenvalues of $A_\pm=d\tilde{F}(\overline{W}_\pm)=dF^1(dF^0)^{-1}(\bar{U}_\pm)$ and $l_i^\pm(x)$/$r_i^\pm(x)$ are left/right eigenvectors of $d\tilde{F}(\overline{W}(x))$ associated with eigenvalues $a_j(x)$ for $x\gtrless 0$ ($a_j(x)\rightarrow a_j^\pm$ as $x\rightarrow \pm\infty$), satisfying
\ba \label{ljkbounds}
&|r_{i}^\pm(x)|\leq C, \qquad
|(\partial/\partial x) r_{i}^\pm(x)|\leq C\gamma e^{-\eta |x|}, \qquad | r_{i}^\pm(x)-r_{i}^\pm(\pm\infty)|\leq C\gamma e^{-\eta |x|},\\
&|l_{i}^\pm(x)|\leq C, \qquad\,
|(\partial/\partial x) l_{i}^\pm(x)|\leq C\gamma e^{-\eta |x|}, \qquad\, | l_{i}^\pm(x)-l_{i}^\pm(\pm\infty)|\leq C\gamma e^{-\eta |x|},
\ea  $\beta_k^\pm>0$ are time-asymptotic scalar diffusion rates $\beta_k^\pm:=(l_k\tilde{B}r_k)^{\pm}$, and 
\begin{equation}\label{Gbounds}
\begin{aligned}
	\,\partial_{x,y}^\alpha & \tilde G(x,t;y)\\
=  
	&\, 
	\big(t^{-|\alpha_x|/2}+ |\alpha_x| e^{-\eta|x|}\big)\big(t^{-|\alpha_y|/2}+ |\alpha_y|\gamma e^{-\eta|y|} \big)\\
&\times\Big( \sum_{k=1}^n r_k^-(x)l_k^-(y) e^{-\eta x^+}
\mathcal{O}(t^{-1/2}e^{-(x-y-a_k^{-} t)^2/(Mt)}) \\
&\qquad+
\sum_{ a_j^{-} < 0,\,a_k^{-} > 0} 
\chi_{\{ |a_k^{-} t|\geq  |y| \}}r_j^-(x)l_k^-(y) 
e^{-\eta x^+}\mathcal{O}(t^{-1/2} e^{-(x-a_j^{-}(t-|y/a_k^{-}|))^2/(Mt)})
\\
&\qquad+\sum_{ a_j^{+}> 0,\,a_k^{-} > 0} 
\chi_{\{ |a_k^{-} t|\geq  |y| \}}r_j^+(x)l_k^-(y) e^{-\eta x^-}
\mathcal{O}(t^{-1/2} e^{-(x-a_j^{+} (t-|y/a_k^{-}|))^2/(Mt)})
\Big)\\&+\mathcal{O}(e^{-\eta(|x-y|+t)}),
\end{aligned}
\end{equation}
for $0\leq |\alpha| \leq 2$, $0\leq |\alpha_y|\leq 1$, some $\eta$, $C$, $M>0$, 
where $x^\pm$ denotes the positive/negative
part of $x$, 
$\gamma$ is 1 for undercompressive profiles and 
0 otherwise. 
Symmetric estimates hold for $y\geq 0$.
\end{proposition}

\br\label{degrmk}
The difference between undercompressive and Lax estimates, alluded to in Section \ref{s:disc},
is precisely the term $|\alpha_y|e^{-\eta|y|}$ appearing in \eqref{Gbounds} for undercompressive
($\gamma=1$) but not Lax ($\gamma=0$) type shocks.
One may check that this leads to nonlinear interaction terms not compatible with the argument of \cite{MZ2}.
\er

{}From \eqref{e} and \eqref{ljkbounds}, we obtain by straightforward
calculation (see \cite{MZ1}) the bounds
\begin{equation}\label{ebounds}
\begin{aligned}
|e(y,t)|&\leq C\chi_{_{t\geq 1}}\sum_{a_i^->0}
  \left({\rm errfn}\bigg(\frac{y+a_i^{-}t}{\sqrt{4\beta_i^{-}t}}\bigg)
  -{\rm errfn}\bigg(\frac{y-a_i^{-}t}{\sqrt{4\beta_i^-t}}\bigg)\right),\\
|e (y,t) - e (y,+\infty)| &\leq C \chi_{_{t\geq 1}}{\rm errfn} \Big(\frac{|y|-at}{M\sqrt{t}}\Big), 
\quad \text{for some }\, a>0, \\
|\partial_t  e(y,t)|&\leq C \chi_{_{t\geq 1}} t^{-1/2} \sum_{a_i^->0} e^{-|y+a_i^-t|^2/(Mt)},\\
|\partial_y  e(y,t)|&\leq C\chi_{_{t\geq 1}} t^{-1/2} \sum_{a_i^->0} e^{-|y+a_i^-t|^2/(Mt)}\\
&\quad +
C\chi_{_{t\geq 1}}\gamma e^{-\eta|y|}
  \left({\rm errfn}\bigg(\frac{y+a_i^{-}t}{\sqrt{4\beta_i^{-}t}}\bigg)
  -{\rm errfn}\bigg(\frac{y-a_i^{-}t}{\sqrt{4\beta_i^{-}t}}\bigg)\right),\\
|\partial_y e (y,t) - \partial_y e(y,+\infty)|
 &\leq  C \chi_{_{t\geq 1}}t^{-1/2} \sum_{a_i^->0} e^{-|y+a_i^-t|^2/(Mt)} \\
|\partial_{yt}e(y,t)|&\leq C\chi_{_{t\geq 1}}
(t^{-1}+\gamma t^{-1/2}e^{-\eta|y|}) \sum_{a_i^->0} e^{-|y+a_i^-t|^2/(Mt)}\\
\end{aligned}
\end{equation}
for $y\leq 0$, and symmetrically for $y\geq  0$, where $\gamma$ as above
is $1$ for undercompressive profiles and zero otherwise. 

\subsection{Auxiliary estimates}\label{s:aux}
To prove the key vertical estimate \eqref{vert}, we shall need the following ``Strichartz-type'' bounds similar
to those established in \cite[Lemma 8.8]{YZ}.

\begin{proposition}[Auxiliary integral bounds]\label{auxprop}
	For $\tilde{G}$ and $H$ terms, there hold:
\be\label{aux1}
\int_{0}^{t} \left| \int_{-\infty}^{+\infty} (\tilde{G}+H)(x,s;y)f(y)dy\right| \, ds \leq C |f|_{L^1\cap L^\infty},
\ee
\ba
\label{aux2}
\int_0^t (1+s)^{-1/2}\left|\int_0^s \int_{-\infty}^{+\infty}\tilde{G}_y(x,s-\tau ;y)f(y,\tau)dy  \, d\tau\right| \,ds
\leq C\int_0^t  (1+s)^{-1/2 +\eps } |f(\cdot, s)|_{L^{2}} ds,
\ea
and
\ba
\label{aux3}
\int_0^t (1+s)^{-1/2}\left|\int_0^s  \int_{-\infty}^{+\infty}H(x,s-\tau ;y)f(y,\tau)dy\, d\tau\right| \, ds
\leq C\int_0^t  (1+s)^{-1/2 } |f(\cdot, s)|_{L^\infty } ds,
\ea
for all $t\geq  0$ and $x\gtrless 0$, any $\eps>0$, and some $C=C(\eps)>0$. 
\end{proposition}

\begin{proof}
	For the part contributed by $\tilde{G}$ in the estimate \eqref{aux1}, changing the order of integration, we obtain
$$
\begin{aligned}
	\int_{0}^{t} \left| \int_{-\infty}^{+\infty} \tilde{G}(x,s;y)f(y)dy\right| \, ds &\leq
\int_{0}^{t}  \int_{-\infty}^{+\infty} |\tilde{G}(x,s;y)||f(y)| dy \, ds \\
&= 
\int_{-\infty}^{+\infty}  |f(y)| \int_{0}^{t}  |\tilde{G}(x,s;y)| \, ds \, dy 
 \leq C |f|_{L^1},
\end{aligned}
$$
where the bound $\int_{0}^{t}  |\tilde{G}(x,s;y)| \, ds\leq C$ may be verified 
as in \cite[Lemma 8.8]{YZ}, by integrating a sum $\theta:=\sum_j \theta_j$ of Gaussian kernels 
	\be\label{theta}
	\theta_j(z,s):=(s)^{-1/2}e^{-(z-a_j s)^2/b_j s}
	\ee
	moving with nonzero speeds $a_j \neq 0$.\footnote{In \cite{YZ} there was only a single Gaussian kernel, 
	but the same estimates hold, by superposition, for the sum.}
	The estimate of the part contributed by $H$ can be seen by using 
	$$\left|\int_{-\infty}^{+\infty}H(x,s;y)f(y)dy\right|\leq e^{-\eta s}|f|_{L^\infty},$$
	and then integrating in time. For the bound \eqref{aux2}, based on the estimate
\ba
\left|\label{test}
\int_1^{t}s^{-1-\eps} e^{-\frac{(as-|\cdot|)^2}{bs}} \, ds\right|_{L^2}\leq C, \quad \text{ for $a\neq 0$, $C=C(\eps)$ independent of $t$,}  
\ea
obtained in \cite[(8.25)]{YZ} and the estimate $\tilde G_y=\theta(x-y,t)\mathcal{O}(t^{-1/2}+e^{-\eta |y|})$, the lefthand side of \eqref{aux2} can be bounded by 
$$
\begin{aligned}
&C\int_0^t \int_0^s \int_{-\infty}^{+\infty}(1+s)^{-1/2}(s-\tau)^{-1/2}\theta(x-y,s-\tau)|f(y,\tau)|dy  \, d\tau \,ds\\
+&C\int_0^t \int_0^s \int_{-\infty}^{+\infty}(1+s)^{-1/2}e^{-\eta|y|}\theta(x-y,s-\tau)|f(y,\tau)|dy  \, d\tau \,ds
\end{aligned}$$
where the first term can be bounded by the right hand side of \eqref{aux2} as in \cite[Lemma 8.8]{YZ}, and the second term can be estimated as
$$
\begin{aligned}
&\,C\int_0^t \int_0^s \int_{-\infty}^{+\infty}(1+s)^{-1/2}e^{-\eta|y|}\theta(x-y,s-\tau)|f(y,\tau)|dy  \, d\tau \,ds\\
=&\,C\int_0^t\int_{-\infty}^{+\infty}e^{-\eta|y|}|f(y,\tau)| \int_\tau^t (1+s)^{-1/2}\theta(x-y,s-\tau)ds  \,dy \,d\tau\\
\leq&\,C\int_0^t\int_{-\infty}^{+\infty}e^{-\eta|y|}|f(y,\tau)|(1+\tau)^{-1/2} \int_\tau^t \theta(x-y,s-\tau)ds  \,dy \,d\tau\\
\leq&\,C\int_0^t(1+\tau)^{-1/2}\int_{-\infty}^{+\infty}e^{-\eta|y|}|f(y,\tau)| \,dy \,d\tau\leq C\int_0^t(1+\tau)^{-1/2} |f(y,\tau)|_{L^2}\,d\tau.
\end{aligned}
$$
Finally, the bound \eqref{aux3} can be shown by the same calculation as  in \cite[(8.24)]{YZ}.
\end{proof}

\subsection{Nonlinear damping estimate}\label{s:damping}
We recall, finally, the following estimate established in \cite{Z1}.

\begin{proposition}[Nonlinear damping estimate \cite{Z1}]\label{damping}
Under the hypotheses of Theorem \ref{main1}
let $w_0\in H^s$, $s\geq  4$, and suppose that, for $0\leq t\leq T$, both the supremum of $ |\dot\delta|$ and the 
$W^{s-1,\infty}$ norm of the solution 
$w$ of \eqref{perteq} remain sufficiently small.
Then, for all $0\leq t\leq T$, and some $\nu>0$,
\be 
|w(t)|_{H^s}^2\leq C |w_0|^2_{H^s}e^{-\nu t}
+C\int_0^t e^{-\nu (t-\tau )}(|w|_{L^2}^2+ |\dot\delta|^2)(\tau)\, d\tau.
\label{ebounds2}
\ee
\end{proposition} 

\section{Nonlinear stability of Lax or undercompressive profiles}\label{uc}
We are now ready to carry out our nonlinear stability analysis.
Defining 
\begin{equation}
  \delta (t)=-\int^\infty_{-\infty}e(y,t) w_0(y)\,dy  +\int^t_0\int^{+\infty}_{-\infty} e_{y}(y,t-s)(Q(w,w_x)+
  \dot \delta\, w)(y,s) dy ds, 
 \label{delta}
\end{equation}
following \cite{ZH}, \cite{Z1}, \cite{MZ1}--\cite{MZ2},
where $e$ is defined as in \eqref{e} (that is, $e=\sum_i e_i$),
and recalling the decomposition $G=H+E+\tilde G$,
we obtain finally the {\it reduced equations}:
\ba 
  w(x,t)=&\int^\infty_{-\infty} \tilde G(x,t;y)w_0(y)\,dy-\int^t_0\int^\infty_{-\infty}\tilde G_y(x,t-s;y)(Q(w,w_x)+
  \dot \delta w)(y,s) dy \, ds\\
 &+\int^\infty_{-\infty} H(x,t;y)w_0(y)\,dy-\int^t_0\int^\infty_{-\infty} H(x,t-s;y)(Q_y(w,w_x)+
  \dot \delta w_y)(y,s) dy \, ds
  \label{veq}
\ea 
and, differentiating (\ref{delta}) with respect to $t$,
and observing that 
$e_y (y,s)\rightharpoondown 0$ as $s \to 0$, as the difference of 
approaching heat kernels:
\begin{equation}
  \dot \delta (t)=-\int^\infty_{-\infty}e_t(y,t) w_0(y)\,dy
  +\int^t_0\int^{+\infty}_{-\infty} e_{yt}(y,t-s)(Q(w,w_x)+
  \dot \delta w)(y,s)\,dy\,ds. 
\label{deltadot}
\end{equation}
\medskip
Define
\ba
 \zeta(t):= \sup_{0\leq s \leq t, 2\leq p\leq \infty}
  \Big(  &\big(|w(\cdot, s)|_{{}_{L^p}}+|w_x(\cdot, s)|_{{}_{L^p}}\big)(1+s)^{\frac{1}{2}(1-\frac1p)}\\
  &+|\delta (s)|+ |\dot\delta (s)|(1+s)^{\frac12} +\int_0^s(1+\tau)^{-1/2}|w(y,\tau)|d\tau\Big).
 \label{zeta2}
\ea

Then, we have the following key estimate, from which nonlinear stability readily follows.

\bl\label{zetalem}
For all $t\geq  0$ for which a solution $w$ exists with
$\zeta(t)$ uniformly bounded by some fixed, sufficiently small constant,
there holds
\be\label{zetaclaim}
 \zeta(t) \leq C(|w_0|_{{}_{L^1 \cap H^4}} + \zeta(t)^2).
 \ee
\el
\begin{proof}
It suffices to show, for $0\leq s\leq t$
$$
\begin{aligned}
&\big(|w(\cdot, s)|_{{}_{L^p}}+|w_x(\cdot, s)|_{{}_{L^p}}\big)(1+s)^{\frac{1}{2}(1-\frac1p)}
  +|\delta (s)|+ |\dot\delta (s)|(1+s)^{\frac12} +\int_0^s(1+\tau)^{-1/2}|w(y,\tau)|d\tau\\
  \leq&\, C(|w_0|_{{}_{L^1 \cap H^4}} + \zeta(t)^2)
.
\end{aligned}
$$

({\it $|\delta|$ bound.}) 
By \eqref{delta}, we have that $\delta(s)$ may
be split into 
\ba 
\label{delta12}
  \delta (s)&=-\int^\infty_{-\infty}e(y,s) w_0(y)\,dy  +\int^s_0\int^{+\infty}_{-\infty} e_{y}(y,s-\tau)(Q(w,w_x)+
  \dot \delta\, w)(y,\tau) dy d\tau\\
  &=:\delta_1 (s)+\delta_2 (s).
\ea 
By  \eqref{ebounds}[i], we readily see $|\delta_1(s)|\leq |w_0|_{L^1}$. To bound $|\delta_2(s)|$, splitting the integral in $y$ as $\int_{-\infty}^0+\int_0^{+\infty}$ and by \eqref{ebounds}[iv], the integral on $y<0$ can be bounded by
\ba 
\label{delta2bound}
&\int^s_0\int^0_{-\infty} \left|e_{y}(y,s-\tau)(Q(w,w_x)+\dot \delta w)(y,\tau)\right|\,dy\,d\tau\\
\leq&\int^s_0\int^0_{-\infty} C
(s-\tau)^{-1/2}\sum_{a_i^->0} e^{-|y+a_i^-(s-\tau)|^2/(M(s-\tau))}\left|(Q(w,w_x)+\dot \delta w)(y,\tau)\right|\,dy\,d\tau\\&+\int^s_0\int^0_{-\infty} C e^{-\eta|y|}\sum_{a_i^->0}
  \left({\rm errfn}\bigg(\frac{y+a_i^{-}(s-\tau)}{\sqrt{4\beta_i^{-}(s-\tau)}}\bigg)
  -{\rm errfn}\bigg(\frac{y-a_i^{-}(s-\tau)}{\sqrt{4\beta_i^{-}(s-\tau)}}\bigg)\right)\\&\qquad\qquad\;\;\,\times\left|(Q(w,w_x)+\dot \delta w)(y,\tau)\right|\,dy\,d\tau.
\ea 
Applying the Cauchy-Schwarz inequality, the first integral can be bounded by $$C\int_0^s (s-\tau)^{-1/4}(1+\tau)^{-3/4}\zeta(t)^2d\tau\leq C\zeta(t)^2.$$
Switching the order of integration, the second integral can be bounded by
\ba
\label{delta2bound2}
&\int^s_0\int^0_{-\infty} C e^{-\eta|y|}
  \left|(Q(w,w_x)+\dot \delta w)(y,\tau)\right|\,dy\,d\tau\\
=&\int^0_{-\infty} \int_0^s C e^{-\eta|y|}
  \left|(Q(w,w_x)+\dot \delta w)(y,\tau)\right|\,d\tau\,dy\\
  \leq &
  \int^0_{-\infty} \int_0^s C e^{-\eta|y|}
  (1+\tau)^{-1/2}\zeta(t)|w(y,\tau)|\,d\tau\,dy\leq \int^0_{-\infty} C e^{-\eta|y|}
 \zeta(t)^2\,dy\leq C \zeta(t)^2
\ea
where we have used $\int_0^s(1+\tau)^{-1/2}|w(y,\tau)|d\tau<\zeta(t)$.

({\it $|\dot \delta|$ bound.}) 
By \eqref{deltadot}, we have that $\dot \delta(s)$ may
be split into 
$$
\begin{aligned}
\dot \delta (s)&=-\int^\infty_{-\infty}e_t(y,s) w_0(y)\,dy+\int^s_0\int^{+\infty}_{-\infty} e_{yt}(y,s-\tau)(Q(w,w_x)+\dot \delta w)(y,\tau)\,dy\,d\tau\\
&=:\dot\delta_{1} (s)+\dot\delta_{2}(s).
\end{aligned}
$$
By \eqref{ebounds}[iii], we readily see $(1+s)^{1/2}|\dot\delta_1(s)|\leq C|w_0|_{L^1\cap L^\infty}$. To bound $|\dot\delta_2(s)|$, splitting the integral in $y$ as $\int_{-\infty}^0+\int_0^{+\infty}$ and by \eqref{ebounds}[vi], the integral on $y<0$ can be bounded by
$$
\begin{aligned}
&\int^s_0\int^0_{-\infty} \left|e_{yt}(y,s-\tau)(Q(w,w_x)+\dot \delta w)(y,\tau)\right|\,dy\,d\tau\\
\leq&\int^s_0\int^0_{-\infty} C
(s-\tau)^{-1}\sum_{a_i^->0} e^{-|y+a_i^-(s-\tau)|^2/(M(s-\tau))}\left|(Q(w,w_x)+\dot \delta w)(y,\tau)\right|\,dy\,d\tau\\&+\int^s_0\int^0_{-\infty} C (s-\tau)^{-1/2}e^{-\eta|y|} \sum_{a_i^->0} e^{-|y+a_i^-(s-\tau)|^2/(M(s-\tau))}\left|(Q(w,w_x)+\dot \delta w)(y,\tau)\right|\,dy\,d\tau.
\end{aligned}
$$
Applying the Cauchy-Schwarz inequality, the first integral can be bounded by $$C\int_0^s (s-\tau)^{-3/4}(1+\tau)^{-3/4}\zeta(t)^2d\tau\leq C(1+s)^{-1/2}\zeta(t)^2.$$
To bound the second integral, we have
\ba
&\int^s_0\int^0_{-\infty} (s-\tau)^{-1/2}e^{-\eta|y|} e^{-|y+a_i^-(s-\tau)|^2/(M(s-\tau))}\left|(Q(w,w_x)+\dot \delta w)(y,\tau)\right|\,dy\,d\tau\\
=&\int^s_0\int^{-a_i^-(s-\tau)/2}_{-\infty} (s-\tau)^{-1/2}e^{-\eta|y|} e^{-|y+a_i^-(s-\tau)|^2/(M(s-\tau))}\left|(Q(w,w_x)+\dot \delta w)(y,\tau)\right|\,dy\,d\tau\\
&+\int^s_0\int_{-a_i^-(s-\tau)/2}^0 (s-\tau)^{-1/2}e^{-\eta|y|} e^{-|y+a_i^-(s-\tau)|^2/(M(s-\tau))}\left|(Q(w,w_x)+\dot \delta w)(y,\tau)\right|\,dy\,d\tau\\
\leq&\int^s_0e^{-\eta a_i^-(s-\tau)/2}\int^{-a_i^-(s-\tau)/2}_{-\infty} (s-\tau)^{-1/2} e^{-|y+a_i^-(s-\tau)|^2/(M(s-\tau))}\left|(Q(w,w_x)+\dot \delta w)(y,\tau)\right|\,dy\,d\tau\\
&+\int^s_0\int_{-a_i^-(s-\tau)/2}^0 (s-\tau)^{-1/2}e^{-\eta|y|} e^{-(a_i^-)^2(s-\tau)/(4M)}\left|(Q(w,w_x)+\dot \delta w)(y,\tau)\right|\,dy\,d\tau\\
\leq&\int^s_0Ce^{-\eta a_i^-(s-\tau)/2}\left|(Q(w,w_x)+\dot \delta w)(\cdot,\tau)\right|_{L^\infty}\,d\tau\\
&+\int^s_0C(s-\tau)^{-1/2} e^{-(a_i^-)^2(s-\tau)/(4M)}\left|(Q(w,w_x)+\dot \delta w)(\cdot,\tau)\right|_{L^\infty}\,d\tau\\
\leq&\int^s_0C e^{-\bar\eta(s-\tau)}(1+\tau)^{-1}\zeta(t)^2\,d\tau+\int^s_0C e^{-\bar\eta(s-\tau)}(s-\tau)^{-1/2}(1+\tau)^{-1}\zeta(t)^2\,d\tau\\
\leq&\, C(1+s)^{-1}\zeta(t)^2+\int^s_0C (s-\tau)^{-3/4}(1+\tau)^{-1}\zeta(t)^2\,d\tau\leq C(1+s)^{-1/2}\zeta(t)^2.
\ea
The integral on $y>0$ can be similarly estimated. Hence $(1+s)^{1/2}|\dot\delta(s)|\leq C(|w_0|_{L^1\cap H^s}+\zeta(t)^2)$. 

({\it $|w|_{L^p}$ bound.})
To shorten the writing, we will use $\theta:=\sum_j \theta_j$ as in \eqref{theta} to denote a sum of
Gaussian kernels $\theta_j$ moving with nonzero speeds $a_j$.

The equation \eqref{veq} of $w$ may be split into
\ba 
\label{wsplitting}
 w(x,s)&=\int^\infty_{-\infty} \tilde G(x,s;y)w_0(y)\,dy-\int^s_0\int^\infty_{-\infty}\tilde G_y(x,s-\tau;y)(Q(w,w_x)+
  \dot \delta w)(y,\tau) dy \, d\tau\\
  &\;\;\;\,+\int^\infty_{-\infty} H(x,s;y)w_0(y)\,dy-\int^s_0\int^\infty_{-\infty} H(x,s-\tau;y)(Q_y(w,w_x)+
  \dot \delta w_y)(y,\tau) dy \, d\tau\\
  &=:w_1(x,s)+w_2(x,s)+w_3(x,s)+w_4(x,s).
\ea 
The $L^p$-norm of the $w_1$ term can be bounded by
\be 
\label{u1bound}
|w_1(\cdot,s)|_{L^p}\leq \int^\infty_{-\infty}  |\tilde G(\cdot,s;y)|_{L^p}|w_0(y)|\,dy\leq C (1+s)^{-\frac{1}{2}(1-\frac{1}{p})}|w_0|_{L^1}.
\ee
The $L^p$-norm of the $w_2$ term can be split and bounded by
\ba 
\label{Lpu2}
|w_2(\cdot,s)|_{L^p}\leq & \int^s_0C(s-\tau)^{-\frac{1}{2}(\frac{3}{2}-\frac{1}{p})}|(Q(w,w_x)+
  \dot \delta w)(\cdot,\tau)|_{L^2} d\tau\\
  &+\left|\int^s_0\int^\infty_{-\infty} Ce^{-\eta|y|}\theta(x,s-\tau;y)(Q(w,w_x)+
  \dot \delta w)(y,\tau) dy\, d\tau\right|_{L^p}
\ea 
where the first term can be bounded by  $\int^s_0C(s-\tau)^{-\frac{1}{2}(\frac{3}{2}-\frac{1}{p})}(1+\tau)^{-\frac{3}{4}}\zeta(t)^2 d\tau\leq C (1+s)^{-\frac{1}{2}(1-\frac{1}{p})}\zeta(t)^2$ and the second term can be bounded by the sum of
$$
\begin{aligned}
&C\left|\int_{s/2}^s\int^\infty_{-\infty} e^{-\eta|y|}\theta(x,s-\tau;y)(Q(w,w_x)+
  \dot \delta w)(y,\tau) dy\, d\tau\right|_{L^p}\\
 \leq&\,C\int_{s/2}^s\left|\int^\infty_{-\infty} e^{-\eta|y|}\theta(x,s-\tau;y)(Q(w,w_x)+
  \dot \delta w)(y,\tau) dy\right|_{L^p}\, d\tau\\
  \leq&\,C\int_{s/2}^s |\theta(\cdot,s-\tau)|_{L^p}|e^{-\eta|\cdot|}|_{L^1}|(Q(w,w_x)+
  \dot \delta w)(\cdot,\tau)|_{L^\infty} \, d\tau\\
  \leq&\,C\int_{s/2}^sC(s-\tau)^{-\frac{1}{2}(1-\frac{1}{p})}(1+\tau)^{-1}\zeta(t)^2d\tau\leq C(1+s/2)^{-1}\zeta(t)^2\int_{s/2}^s(s-\tau)^{-\frac{1}{2}(1-\frac{1}{p})}d\tau\\
  \leq&\, C(1+s)^{-\frac{1}{2}(1-\frac{1}{p})}\zeta(t)^2
\end{aligned}
$$
and
$$
\begin{aligned}
&\,C\left|\int_0^{s/2}\int^\infty_{-\infty} e^{-\eta|y|}\theta(x,s-\tau;y)(Q(w,w_x)+
  \dot \delta w)(y,\tau) dy\, d\tau\right|_{L^p}\\
 =&\,C
	\left|\int^\infty_{-\infty}e^{-\eta|y|}\int_0^{s/2} \theta(x,s-\tau;y)(Q(w,w_x)+
  \dot \delta w)(y,\tau) d\tau\,dy\right|_{L^p}\\
  \leq &\,C
	\int^\infty_{-\infty}e^{-\eta|y|}\int_0^{s/2} (s-\tau)^{-\frac{1}{2}(1-\frac{1}{p})}\big(|w(\cdot,\tau)|_{L^\infty}+|w_x(\cdot,\tau)|_{L^\infty}+|\dot\delta|(\tau)\big) |w(y,\tau)| d\tau\,dy\\
  \leq&\,C
	s^{-\frac{1}{2}(1-\frac{1}{p})}\zeta(t)\int^\infty_{-\infty}e^{-\eta|y|}\int_0^{s/2} (1+\tau)^{-1/2}|w(y,\tau)| d\tau\,dy\\
   \leq &\,C(s)^{-\frac{1}{2}(1-\frac{1}{p})}\zeta(t)^2.
\end{aligned}
$$
{\it Note that we have used here in a crucial way
the vertical estimate $\int_0^{s/2} (1+\tau)^{-1/2}|w(y,\tau)| d\tau\leq C\zeta(s)$
to deal with outgoing characteristic modes $\theta$. (In the case, as in \cite{YZ}, that all
modes $\theta$ move inward toward the shock, the contribution of the second term is time-exponentially
small, and may be neglected.)}
In the case $s\leq 1$, the second term in \eqref{Lpu2} is majorized by (a multiple of) the first term in 
\eqref{Lpu2}, and so may be estimated in the same way to obtain a bound of $C\zeta(t)^2$; thus, we may replace
the last bound, $C(s)^{-\frac{1}{2}(1-\frac{1}{p})}\zeta(t)^2$, by
$C(1+s)^{-\frac{1}{2}(1-\frac{1}{p})}\zeta(t)^2$, as required, removing the apparent singularity as
$s\to 0$.
The $L^p$-norm of the $w_3$ term can be bounded by 
\be 
\label{u3bound}
\left|\int^\infty_{-\infty} H(\cdot,s;y)w_0(y)\,dy\right|_{L^p}\leq e^{-\eta s}|w_0|_{L^p}.
\ee 
The $L^p$-norm of the $w_4$ term can be bounded by 
\ba 
\label{u4bound}
&\left|\int_0^s\int^\infty_{-\infty} H(x,s-\tau;y)(Q_y(w,w_x)+\dot\delta w_y)(y,\tau)\,dy\,d\tau\right|_{L^p}\\
\leq &\,C\int_0^s e^{-\eta(s-\tau)}|Q_y(w,w_x)+\dot\delta w_y)(\cdot,\tau)|_{L^p}\,d\tau\\
\leq&\,C\int_0^s e^{-\eta(s-\tau)}\left((1+\tau)^{-1/2}\zeta(t)|w_y|_{L^p}+|w_{yy}|_{L^\infty}|w|_{L^p}\right)\,d\tau
\ea 
where, by definition \eqref{zeta2}, 
\ba 
\label{u4bound1}
\int_0^s e^{-\eta(s-\tau)}(1+\tau)^{-1/2}\zeta(t)|w_y|_{L^p}\,d\tau\leq&\int_0^s e^{-\eta(s-\tau)}(1+\tau)^{-\frac{1}{2}(2-\frac{1}{p})}\zeta(t)^2\,d\tau\\
\leq&(1+s)^{-\frac{1}{2}(2-\frac{1}{p})}\zeta(t)^2
\ea 
and by the estimate \eqref{ebounds4}, 
\ba 
\label{u4bound2}
&\int_0^s e^{-\eta(s-\tau)}|w_{yy}|_{L^\infty}|w|_{L^p}\,d\tau\\
\leq& \,C \int_0^s e^{-\eta(s-\tau)}\left(|w_0|_{H^3}+\zeta(t)\right)(1+\tau)^{-\frac{1}{4}}(1+\tau)^{-\frac{1}{2}(1-\frac{1}{p})}\zeta(t)\,d\tau\\
\leq& \,C|w_0|_{H^3}\zeta(t)(1+s)^{-\frac{1}{2}(\frac{3}{2}-\frac{1}{p})}+C\zeta(t)^2(1+s)^{-\frac{1}{2}(\frac{3}{2}-\frac{1}{p})}\\
\leq& \,C|w_0|_{H^3}^2(1+s)^{-\frac{1}{2}(\frac{3}{2}-\frac{1}{p})}+C\zeta(t)^2(1+s)^{-\frac{1}{2}(\frac{3}{2}-\frac{1}{p})}\\
\leq& \,C|w_0|_{H^3}(1+s)^{-\frac{1}{2}(\frac{3}{2}-\frac{1}{p})}+C\zeta(t)^2(1+s)^{-\frac{1}{2}(\frac{3}{2}-\frac{1}{p})}.
\ea 
({\it $|w_x|_{L^p}$ bound.})
By \eqref{veq}, we have that $w_x$ may be split into
$$
\begin{aligned}
 w_x(x,s)&=\int^\infty_{-\infty} \tilde G_x(x,s;y)w_0(y)\,dy-\int^s_0\int^\infty_{-\infty}\tilde G_{xy}(x,s-\tau;y)(Q(w,w_x)+
  \dot \delta w)(y,\tau) dy \, d\tau\\
  &+\partial_x\int^\infty_{-\infty} H(x,s;y)w_0(y)\,dy-\int^s_0\partial_x\int^\infty_{-\infty} H(x,s-\tau;y)(Q_y(w,w_x)+
  \dot \delta w_y)(y,\tau) dy \, d\tau\\
  &=:w_{1x}(x,s)+w_{2x}(x,s)+w_{3x}(x,s)+w_{4x}(x,s).
  \end{aligned}
$$
Following a similar estimate as \eqref{u1bound}, the $L^p$-norm of $w_{1x}$ can be bounded by 
$$
C(1+s)^{-\frac{1}{2}(1-\frac{1}{p})}|w_0|_{L^1\cap L^p}.
$$
To bound the $L^p$-norm of $w_{2x}$, we split the time integral as $\int_0^s=\int_0^{s-1}+\int_{s-1}^s$. 
And, note that for $s-\tau\geq  1$, $\tilde{G}_{xy}$ satisfies bounds at least as good as those for $\tilde{G}_y$. 
Hence,
$$
\left|\int^{s-1}_0\int^\infty_{-\infty}\tilde G_{xy}(x,s-\tau;y)(Q(w,w_x)+
  \dot \delta w)(y,\tau) dy \, d\tau\right|_{L^p}\leq C(1+s)^{-\frac{1}{2}(1-\frac{1}{p})}\zeta(t)^2.
$$  
The remaining integral on $\int_{s-1}^s$, by Young's convolution inequality and triangular inequality, can be estimated as
$$
\begin{aligned}
&\left|\int_{s-1}^s\int^\infty_{-\infty}\tilde G_{xy}(x,s-\tau;y)(Q(w,w_x)+
  \dot \delta w)(y,\tau) dy \, d\tau\right|_{L^p}\\
  =&\left|\int_{s-1}^s\int^\infty_{-\infty}\tilde G_{x}(x,s-\tau;y)(Q(w,w_x)+
  \dot \delta w)_y(y,\tau) dy \, d\tau\right|_{L^p}\\
  \leq& \int_{s-1}^s\big|(s-\tau)^{-1/2}\theta(\cdot,s-\tau)\big|_{L^1}|(Q(w,w_x)+
  \dot \delta w)_y(\cdot,\tau)|_{L^p}d\tau
  \\&+\Big|e^{-\eta|\cdot|}\int_{s-1}^s\int_{-\infty}^\infty \theta(\cdot-y,s-\tau)|(Q(w,w_x)+
  \dot \delta w)_y(y,\tau)|dyd\tau\Big|_{L^p}\\
  \leq& \,C\int_{s-1}^s (s-\tau)^{-1/2}(1+\tau)^{-\frac{1}{2}(\frac{3}{2}-\frac{1}{p})}\big(|w_0|_{H^3}+\zeta(t)^2\big)d\tau+\int_{s-1}^s|(Q(w,w_x)+
  \dot \delta w)_y|_{L^\infty}d\tau\\
  \leq& \,C(1+s)^{-\frac{1}{2}(\frac{3}{2}-\frac{1}{p})}\big(|w_0|_{H^3}+\zeta(t)^2\big)
  \end{aligned}
$$
where we have used the Sobolev bound $|w|_{W^{2,\infty}}\leq C |w|_{H^3}$ and estimate \eqref{ebounds4}.

The $L^p$ norm of $w_{3x}$ can be bounded by $e^{-\eta s}|w_0|_{W^{1,p}}$ and hence by $e^{-\eta s}|w_0|_{H^s}$, $s\geq  2$. The $L^p$ norm of $w_{4x}$ can be bounded by
$$
\begin{aligned}
|w_{4x}(\cdot,s)|_{L^p}\leq& \int_0^s Ce^{-\eta(s-\tau)} \left(|\dot\delta w_{xx}|_{L^p}+|w_x w_x|_{L^p}+|w w_{xx}|_{L^p}+|w_x w_{xx}|_{L^p}+|w w_{xxx}|_{L^p}\right)d\tau\\
\leq &\int_0^s Ce^{-\eta(s-\tau)} \left((1+\tau)^{-1/2}\zeta(t)|w|_{H^4}+(1+\tau)^{-\frac{1}{2}(2-\frac{1}{p})}\zeta(t)^2\right)d\tau\\
\leq &\int_0^s Ce^{-\eta(s-\tau)} \left((1+\tau)^{-1/2}\zeta(t)|w|_{H^4}+(1+\tau)^{-\frac{1}{2}(2-\frac{1}{p})}\zeta(t)^2\right)d\tau\\
\leq &\int_0^s Ce^{-\eta(s-\tau)} \left((1+\tau)^{-3/4}\zeta(t)\big(|w_0|_{H^4}+\zeta(t)\big)+(1+\tau)^{-\frac{1}{2}(2-\frac{1}{p})}\zeta(t)^2\right)d\tau\\
\leq &\, (1+s)^{-\frac{1}{2}(1-\frac{1}{p})}C(|w_0|_{H^4}+\zeta(t)^2).
\end{aligned}
$$
({\it Vertical estimate.}) To control integral $\int_0^s (1+\tau)^{-1/2}|w(y,\tau)|d \tau$, by the decomposition of $w$ \eqref{veq}, we 
use the auxiliary estimates \eqref{aux1}-\eqref{aux3} of Proposition \ref{auxprop}.
Applying these estimates together with the bounds
\ba \label{sourcebounds}
\Big|Q(w,w_x)+\dot\delta w\Big|_{L^2} &\leq C (1+s)^{-1/2} \zeta(t)|w|_{L^2}\leq C(1+s)^{-3/4}\zeta(t)^2,\\ |Q(w,w_x)_y+\dot\delta(t)w_y|_{L^\infty}&\leq C\left(|w w_x|_{L^\infty}+|w_x|_{L^\infty}^2+|ww_{xx}|_{L^\infty}\right)\\
&\leq C(\zeta(t)^2(1+s)^{-\frac{3}{4}}+\zeta(t)(1+s)^{-\frac{1}{2}}|w_0|_{H^3}e^{-\theta s})
\ea
we obtain the vertical estimate
$$\int_0^s(1+\tau)^{-1/2}|w(y,\tau)|d\tau\leq C(|w_{0}|_{L^1\cap H^s}+\zeta^2(t)).
$$
\end{proof}

\begin{proof}[Proof of Theorem \ref{main1}]. 
By the assumption on $dF^0>0$ in {(A1)}, it is equivalent to prove \eqref{mainests} with $w$ in the place of $\tilde{U}(\cdot+\delta(t),t)-\bar{U}(\cdot)$. From Lemma \ref{zetalem}, it follows by continuous induction that, provided 
$|w_0|_{{}_{L^1\cap H^4}} < 1/(4C^2)$, there holds 
\begin{equation}
 \zeta(t) \leq 2C |w_0|_{{}_{L^1\cap H^4}}
 \label{bd}
\end{equation}
for all $t\geq 0$ such that $\zeta$ remains small.
For, by standard short-time existence theory \cite{K}
there exists a solution $w(\cdot,t)\in H^4$
on the open time-interval for which $|w|_{H^4}$ remains bounded and sufficiently small,
and thus $\zeta$ is well-defined and continuous. 
Now, let $[0,T)$ be the maximal interval on which $|w|_{{}_{H^4}}$
remains strictly bounded by some fixed sufficiently small constant $c>0$.

Combined with \eqref{zeta2}, the damping estimate \eqref{ebounds2} yields
\ba 
|w(t)|_{H^s}^2\leq& \,C |w_0|^2_{H^s}e^{-\nu t}
+C\zeta(t)^2\int_0^t e^{-\nu (t-\tau )}(1+\tau)^{-1/2}\, d\tau\\
\leq&\, C |w_0|^2_{H^s}e^{-\nu t}
+C\zeta(t)^2(1+t)^{-1/2},
\label{ebounds3}
\ea
giving
\be 
\label{ebounds4}
|w(t)|_{H^s}\leq C |w_0|_{H^s}e^{-\nu t}+C\zeta(t)(1+t)^{-1/4}\leq C( |w_0|_{H^s}+\zeta(t))(1+t)^{-1/4},
\ee 
for all $0\leq t\leq T$. Hence, taking $s=4$, we obtain
\be
|w(t)|_{H^4}\leq C |w_0|_{H^4}e^{-\nu t}+C\zeta(t)(1+t)^{-1/4} \leq C|w_0|_{L^1\cap H^4}(1+t)^{-1/4},
 \label{2calc}
\ee
and so the solution continues so long as $|w_0|_{L^1\cap H^4}$ is taken sufficiently small.
	
This yields global existence and the claimed 
bound on $|w|_{H^s}$. Meanwhile, the bound \eqref{bd} at once yields the claimed bounds on $|w|_{L^p}$, $|w_x|_{L^p}$, $|\dot\delta|$, and $|\delta|$.
\end{proof}

\end{document}